\documentclass[a4paper,12pt]{amsart}
\usepackage{graphicx}
\usepackage{amsmath}
\usepackage{amssymb,amsfonts}
\usepackage{combelow}

\usepackage{helvet}

\RequirePackage{mathrsfs} \let\mathcal\mathscr

\RequirePackage{mathrsfs} \let\mathcal\mathscr

\numberwithin{equation}{section}

\newtheorem{theorem}{Theorem}[section]
\newtheorem{lemma}[theorem]{Lemma}

\theoremstyle{definition}

\newtheorem{remark}[theorem]{Remark}

\newcommand{\R}{\widehat{\mathbb R}^2}

\newcommand{\sll}{\mathfrak sl(2,\mathbb R)}

\newcommand{\C}{\mathbb C}

\newcommand{\ve}{\varepsilon}

\def\NN{{\mathbb N}}
\def\QQ{{\mathbb Q}}

\def\RR{{\mathbb R}}

\def\ZZ{{\mathbb Z}}

\def\vecnull{{\text{\boldmath$0$}}}

\def\scrO{{\mathcal O}}

\def\fm{\mathfrak{m}}

\def\C{\operatorname{C{}}}

\def\L{\operatorname{L{}}}

\def\SL{\operatorname{SL}}

\def\mod{\operatorname{mod}}

\newcommand{\x}{\mathbf{x}}
\newcommand{\y}{\mathbf{y}}

\renewcommand{\t}{\mathbf{t}}
\newcommand{\z}{\mathbf{z}}

\renewcommand{\v}{\mathbf{v}}

\newcommand{\h}{\mathbf{h}}

\def\SLR{\SL(2,\RR)}

\newcommand{\smatr}[4]{\bigl( \begin{smallmatrix} #1 & #2 \\ #3 & #4 \end{smallmatrix} \bigr) }

\begin{document}

\title{A sparse equidistribution result for $(\SLR/\Gamma_0)^{n}$}
\author{Pankaj Vishe}
\address{
Department of Mathematical Sciences\\
Durham University\\
Durham\\ DH1 3LE\\ United Kingdom}
\email{pankaj.vishe@durham.ac.uk}
\begin{abstract}
Let $G=\SLR^n$, let $\Gamma=\Gamma_0^n$, where $\Gamma_0$ is a co-compact lattice in $\SLR$, let $F(\x)$ be a non-singular quadratic form and let $u(x_1,...,x_n):=\smatr{1}{x_1}{0}1\times...\times \smatr{1}{x_n}{0}1 $ denote unipotent elements in $G$ which generate an $n$ dimensional horospherical subgroup. We prove that in the absence of any local obstructions for $F$, given any $x_0\in G/\Gamma$, the sparse subset $\{u(\x)x_0:\x\in\ZZ^n, F(\x)=0\}$ equidistributes in $G/\Gamma$ as long as $n\geq 481$, independent of the spectral gap of $\Gamma_0$.
\end{abstract}
\maketitle

\section{Introduction}
Let $G$ be a Lie group, let $\Gamma$ be a lattice in $G$ and let $M=G/\Gamma$. Let $U$ be a unipotent subgroup of $G$. Recently there has been an increased interest in understanding the behaviour of sparse arithmetic subsets of $U$ orbits in $M$. It is widely considered that under some reasonable assumptions on $G$ and $\Gamma$, certain discrete arithmetic subsets of dense unipotent orbits should equidistribute in $M$, independent of the choice of the starting points of these orbits.  Let $\Gamma_0\subset G_0=\SLR$ be a co-compact lattice, let $M_0=G_0/\Gamma_0$ and given $x\in\RR$, let $$
u_0(x) := \left(\begin{array}{cc} 1 & x \\ 0 & 1 \end{array}\right)\, 
$$
be the matrices generating a one-parameter horocycle flow on $G_0$. In this case, a conjecture by Shah \cite{Shah94}, later generalised by  Margulis \cite{Margulis00} predicts that ergodic averages of these unipotent trajectories evaluated at polynomial times equidistribute irrespective of their starting points. Namely, that the set $\{u_0(F(n))x_0:n\in\NN\}$ equidistributes in $M_0$ for any polynomial $F$ and any $x_0\in M_0$. This conjecture remains completely open even in the case $F(x)=x^2$. There have been several works which establish results of {\em metric} nature. Namely, they bound the size of the set of initial points which violate this expectation, as seen by works of Bourgain \cite{Bourgain89}, Ubis and Sarnak \cite{SarnakUbis15} and Katz \cite{Katz16} among others. Apart from works of Venkatesh \cite{Venkatesh05}, Tanis and the author \cite{Tanis_Vishe17} and Flaminio, Forni and Tanis \cite{FFT16}, there haven't been many results available which establish such sparse equidistribution results for {\em every} such orbit.

There are several related generalisations that have been studied. Given $n$ commuting ergodic invertible measure preserving one parameter flows $U_1,U_2,\dots, U_n$ on a probability space $(X,\mu)$, let $U(\t):=(U_1(t_1),...,U_n(t_n))$ denote the corresponding flow on the product space $X\times...\times X$. Jones \cite{Jones93} considered ergodic averages of functions on the expanding spherical sub-orbits $U(B_r)x_0$, where $B_r$ denotes the $n$ dimensional sphere $x_1^2+...+x_n^2=r$. Namely, given $f\in \L^p(X\times...\times X)$, \cite{Jones93} considers the spherical averages
\begin{equation}
\label{eq:Spherical}
\int_{\t\in B_r}f(U(\t)x_0)d\sigma_{B_r}(\t),  
\end{equation}
where $\sigma_{B_r}$ denotes an appropriately normalised Haar probability measure on $B_r$. Jones
 proved that the spherical average in \eqref{eq:Spherical} tends to the spatial average of $f$ as $r\rightarrow\infty$ for {\em almost every} $x_0$, as long as $n\geq 3$ and $p\geq n/n-1$. Moreover, Jones in \cite[Theorem 3.1]{Jones93} also attempts a discrete version of this problem. He was able to prove the equidistribution of discrete averages of integer points lying on the annuli
\begin{equation*}
\lim_{k\rightarrow\infty}\#\{\t\in\ZZ^n, r_k\leq |\t|\leq r_k+d_k\}^{-1}\sum_{\substack{\t\in\ZZ^n\\ r_k-d_k\leq |\t|\leq r_k+d_k}}f(U(\t)x_0),
\end{equation*} 
for almost every $x_0$, where $d_kr_k\rightarrow \infty$ and the thickness $d_k$ is bounded, as long as $n\geq 5$ and $p\geq n/n-1 $. This result thus just falls short of being able to cope with the natural discrete analogue: $\{U(\x)x_0: \x\in \ZZ^n\cap B_r\}$ as $r\rightarrow\infty$.  As in the case of $\SLR/\Gamma_0$, it is believed that under suitable conditions on $X=G/\Gamma$, these sparse arithmetic averages of $U$ orbits should equidistribute for every such orbit. In the special case when $X=\RR/\ZZ $ and $U_i(t)(x):=tx\bmod{1}$, Magyar in \cite{Magyar14}, proves the equidistribution the sparse set
$$\{U(\t)x_0\bmod{1}:x_0\in(\RR/\ZZ)^n, \t\in \ZZ^n\cap \{F(\t)=\lambda\}\cap (-P,P)^n\}, $$
where $F$ is a polynomial with a non-singular leading degree form, as $P\rightarrow\infty$ for every {\em Diophantine} initial point $x_0$ as long as $n\gg_{F,\lambda}1$. 

We now state the context in this paper. As before, let $\Gamma_0\subset G_0=\SLR$ be a co-compact lattice.
Let $G=\SLR^n$, let $\Gamma=\Gamma_0\times\cdots\times\Gamma_0$, let $M= G/\Gamma$ and let $d\mu_G=d\mu_{G_0}\times...\times d\mu_{G_0}$, where $d\mu_{G_0}$ denotes the Haar measure on $M_0$ normalised such that $\int_{M_0}d\mu_{G_0}(g)=1$. Let  $$U:=\{u(\x):=u_0(x_1)\times...\times u_0(x_n), \,\,x_1,...,x_n\in \RR\}$$ denote the expanding horospherical subgroup corresponding to the action of a suitable ray of a standard one parameter geodesic flow. The equidistribution of the whole $U$ orbit for every $x_0\in M$ follows from Ratner's equidistribution theorems. In the vein of the aforementioned conjectures by Shah and Margulis, a question by Lindenstrauss on spherical horospheric  averages led Ubis \cite{Ubis17} to investigate analogues of \eqref{eq:Spherical} in this particular setting. Ubis establishes the equidistribution of orbits of the type $U(V)x_0$, where $V$ is any totally curved sub-manifold of $\RR^n$ of {\em low} co-dimension, as long as $n$ is {\em large enough} depending on the spectral gap of $\Gamma_0$ as well as the co-dimension of the manifold. He achieves this by locally approximating pieces of such a manifold by quadratic hypersurfaces and then proving the equidistribution of {\em every} $U$-orbit restricted to a quadric hypersurface in $\RR^n$. 

Here, we work on a natural generalisation of the works of Magyar \cite{Magyar14} and Ubis \cite{Ubis17}.  In particular, we consider the sparse subsets of integer points in $U$ orbits in $M$, which lie on a quadratic hypersurface in $\RR^n$. Let $F(\x)=\x^tL\x\in\ZZ[\x]$ be a smooth quadratic form in $n$ variables defined by an invertible $n\times n$ matrix $L$ with integer entries. We further assume that $F$ has no local obstructions, i.e. that $F(\x)=0$ for some $\x\in\RR^n\setminus\vecnull$ as well as for some $\x\in\QQ_p^n\setminus\vecnull$ for each prime $p$. Given any parameter $P>0$, a standard circle method result (see \cite[Theorem 1]{Birch61} for example) hands us a constant $0<\gamma'\ll 1$ such that the following asymptotic formula holds as long as $n\geq 5$:
\begin{equation}\label{eq:NFPdef}
N_F(P):=\#\{\x\in\ZZ^n:|\x|<P, F(\x)=0\}=C_F P^{n-2}+O(P^{n-2-\gamma'}).
\end{equation}
 The implied constant $C_F>0$ if and only if $F$ has no local obstructions. Here and throughout, we use the notation $A\ll B$ to denote that $A\leq CB$, for some constant $C$. Throughout, our implied constants in $\ll$ are allowed to depend freely on $\Gamma, n$ and  $F$. Any further dependence will be explicitly denoted via adding a subscript to $\ll$. 

Our main goal is to prove the following sparse equidistribution/mixing result:
\begin{theorem}\label{thm:main thm1}
Let $G_0=\SLR$, let $\Gamma_0$ be a co-compact lattice in $G_0$,  let $G=\SLR^n$ and let $\Gamma=\Gamma_0^n$. Then for any non-singular quadratic form $F(\x)\in\ZZ[x_1,...,x_n]$ with no local obstructions, any point $x_0\in G/\Gamma$ and any continuous function $f\in\C(G/\Gamma)$,  we have
\begin{equation*}
\lim_{P\rightarrow\infty }\frac{1}{N_F(P)}\sum_{\substack{\x\in\ZZ^n, |\x|<P\\ F(\x)=0 }}f(u(\x)x_0)=\int f(x) d\mu_G(x),
\end{equation*}
as long as $n\geq 481$.
\end{theorem}
There are two main highlights of this result. Firstly, the equidistribution is established for {\em every} $x_0\in G/\Gamma$. Secondly, our bound for $n$ is independent of the spectral gap of $G_0/\Gamma_0$. We have also not tried to optimise the lower bound $481$ appearing here. It is likely that our bounds may be improved slightly to be able to obtain a better result. A natural limit of the process here would be $n\geq 457=2\times 228+1$, which arises from the second term in our van der Corput bound \eqref{eq:boundminor2}. We must highlight that the situation considered here is significantly different than that of $(\RR/\ZZ)^n$ or the usual circle method setting which leads to the asymptotic formula in \eqref{eq:NFPdef}. Here, we need to consider exponential sums of the type
\begin{equation}\label{eq:Szdef}
S(\alpha):=\sum_{\substack{\x\in\ZZ^n}}w(\x/P)f(u(\x)x_0)e(\alpha F(\x)),
\end{equation}
where $w$ is a suitable compactly supported function on $\RR^n$ and $\alpha$ is a real number. Notice that the extra factor $f(u(\x)x_0)$ appearing here is highly oscillatory, and due to this, the usual analytic techniques break down. We therefore need to lower the degree of $F$ using some sort of differencing, and then use bounds for twisted averages of functions along horocycles. One way to do so is to use van der Corput differencing, which hands us exponential integrals of differenced functions (see \eqref{eq:fhdefn}); which we follow up by applying uniform bounds for twisted horocyclic averages in \cite[Theorem 1,1]{FFT16}, which require taking large, in fact $7+\ve$, number of derivatives of these functions. The final optimisation therefore amounts to a loss by a factor of size $14+\ve$. However, unfortunately, this bound is not enough and we need to invent a new technique, namely,  our alternate bound in Lemma \ref{lem:splitbound}. For further explanation of this technique, we refer the reader to the explanation given following the statement of Theorem \ref{thm:Main thm}. As far as our knowledge, the differencing technique used to obtain Lemma \ref{lem:splitbound} has not yet been used in this setting before. It would be interesting if this bound can be modified to be made to work in the whole {\em minor arcs} regime. In this case, one may be able to obtain the result using much lower number of variables. However, it should be noted that using analytic methods, one may not expect a result as good as $n\geq 5$. Unless one improves upon the work of Flaminio, Forni and Tanis \cite{FFT16}, one would at least require $2\times 6+1=13$ variables (or more realistically $2\times13=26$ variables to allow for a typical loss arising due to differencing).

Theorem \ref{thm:main thm1} can also be seen as a {\em mixing} type result. If $F$ were a diagonal form instead, then as noted by Ubis in a private communication, a simple H\"older inequality type argument applied to the exponential sums in the spirit of techniques used in ternary Goldbach conjecture (see \cite[Lemma 19.4]{IwaKowa04})  can directly establish Theorem \ref{thm:main thm1} as soon as $n\geq 5$. A sketch of this argument will be produced in Remark \ref{rk:diagonal}. The assumption that $\Gamma_0$ is co-compact could also be removed with some more technical work, using finer results in \cite[Theorem 1.1]{FFT16}. We believe that the method in this paper can be suitably modified to obtain a version of Ubis' result \cite{Ubis17} independent of the spectral gap.  In this case, possibly a variant of Lemma \ref{lem:splitbound} itself could be made to work which may lead to requiring a relatively few number of variables.

The strategy used in this paper is rather {\em soft} and is capable of establishing a much more general result than the one stated here. For example, let $X$ be a probability space, and let $U$ be an $n$-dimensional measure preserving flow on $X$, then techniques here may be used to establish the equidistribution of discrete sparse subsets $\{U(\x)x_0:\x\in\ZZ^n\cap\{F(\x)=0\}\}$  (or of its continuous version a.k.a. \cite{Ubis17}), for every $x_0$, as long as one has an effective bound for the twisted averages of the flow $U$ on $X$ and that the dimension $n$ of the flow is {\em large enough}. One relevant application could be to the case where $U$ denotes a full dimensional horospheric flow on $X=\SL(n,\RR)/\SL(n,\ZZ)$ as long as $n$ is large enough with respect to the degree $d$ of the polynomial $F$.

We now move on to the statement of Theorem \ref{thm:Main thm}, our main tool in proving Theorem \ref{thm:main thm1}. Let $x_0$ be an arbitrary fixed point in $G/\Gamma$. In order to use Fourier analytic tools effectively, given any parameter $P\geq 1$, given any $f\in\C^\infty(M)$ and any compactly supported function $w\in\C^\infty_c(\RR^n)$, we will consider the following smooth average 
\begin{equation}\label{eq:sigdef}
\Sigma(P)=\sum_{\substack{\x\in\ZZ^n\\ F(\x)=0}}w(\x/P)f(u(\x)x_0).
\end{equation}
Our main tool for proving Theorem \ref{thm:main thm1} will be provided by Theorem \ref{thm:Main thm} below. It establishes an effective bound for the smooth sum $\Sigma(P)$ for any smooth factorisable function $f\in \C^\infty(G/\Gamma)$ of zero average and a suitably chosen factorisable function $w\in\C^\infty_c((-1,1)^n)$. More explicitly, in Theorem \ref{thm:Main thm}, we assume that $f$ is of the form
\begin{equation}\label{eq:fprod}f(g_1,...,g_n)=\prod_{i=1}^{n} f_i(g_i),\textrm{ where }f_i\in\C^\infty(G_0/\Gamma_0)\textrm{ and }\int_{M_0} f_1(g_1)d\mu_{G_0}(g_1)=0.
\end{equation}
Here, $g_i\in G_0/\Gamma_0$. 

Similarly, we will work with factorisable functions on $\RR^n$. Let $\omega\in \C^\infty_c(\RR)$ be a smooth compactly supported function on $\RR$, whose support is contained in $(-1,1)$ and let \begin{equation}\label{eq:wdef}
w(\x):=\prod_{i=1}^n \omega(x_i).\end{equation}
The implied constants in our final bounds may depend on the measure of the support of $\omega$. The fact that $\omega$ is supported in $(-1,1)$ is only assumed to simplify this dependence in Theorem \ref{thm:Main thm}.

Before we give the statement of Theorem \ref{thm:Main thm}, we must set some notation for various Sobolev norms appearing there. For any function $w(\x)$ in $\C^\infty(\RR^n)$, any $k\in\ZZ_{\geq0}$ and any real number $p\in[1,+\infty]$,
we introduce the standard Sobolev norm $S_{p,k}(w)$, taking values in $\R_{\geq0}\cup\{+\infty\}$, through
\begin{align}\label{eq:normdef}
S_{p,k}(w)=\sum_{j=0}^k\sum_{|\beta|=j}\|\partial_{\x}^\beta w(\x) \|_{\L^p}.
\end{align} 
Here, given $\beta=(k_1,...,k_n)\in\ZZ_{\geq 0}^n$, let $|\beta|=k_1+...+k_n$, and let $\partial_{\x}^\beta:=\partial^{k_1}_{x_1}...\partial^{k_n}_{x_n}$. 

Our norms for functions $f\in \C^\infty(G/\Gamma)$ will be analogous and standard. Let $\{Y, X, Z\}$ be a basis for the Lie algebra $\sll$ given by, 
\begin{equation}
\label{eq:lieder}
Y = \left(\begin{array}{rr}1/2 & 0 \\0 & -1/2\end{array}\right), \ \ X = \left(\begin{array}{rr}0 & 1 \\0 & 0\end{array}\right), \ \  Z = \left(\begin{array}{rr}0 & 0 \\1 & 0\end{array}\right).
\end{equation}
Given $p\in [1,\infty]$ and $k\in\ZZ_{\geq 0}$, by $\|f\|_{\L^p_k}$ we denote the sums of $\L^p$ norms of upto ``$k$-derivatives'' of $f$. To formalise this, let $\mathcal O_k$ be a collection of vectors $D:=(D_1,...,D_n)$, where each co-ordinate $D_i$ is a monomial in
$\{Y_i, X_i, Z_i\}$ such that the total order of all these monomials is at most $k$. Here, $X_i$ (and analogously $ Y_i$ and $Z_i$) denotes the element in the lie algebra of $G/\Gamma$ which contains $X$ in the $i$-th co-ordinate and zero everywhere else, i.e., $X_i=(0,..,0,X,0,...).$ Then we define
\begin{equation}
\label{eq:normdef1}
\|f\|_{\L^p_k} := \sum_{D=(D_1,...,D_n) \in \mathcal O_k} \|D f\|_{\L^p},
\end{equation}
where $Df:=D_1D_2...D_n f$. Upon interpolation as in \cite{Lions}, the above norms can be extended to hold for all $k\in\RR_{\geq 0}$.

We are now set to state Theorem \ref{thm:Main thm}:
\begin{theorem}
\label{thm:Main thm}
There exists an absolute constant $\gamma_0:=\gamma_0(\Gamma)$ such that given any non-singular quadratic form $F(\x)\in\ZZ[x_1,...,x_n]$, any $P\geq 1$, given any $w(\x)\in \C^\infty_c((-1,1)^n)$ satisfying \eqref{eq:wdef} and any $f\in\C^\infty(G/\Gamma)$ satisfying \eqref{eq:fprod}, we have
\begin{equation*}
\left|\Sigma(P)\right|\ll S_{\infty,9n}(w)\|f\|_{\L^\infty_{9n+1}}P^{n-2-\gamma_0},
\end{equation*}
as long as $n\geq 481$.
\end{theorem}
As before, we have not tried to optimise the Sobolev norms as well as the number $481$ appearing in Theorem \ref{thm:Main thm}.

Let us give an overview of the method that will be used to prove Theorem \ref{thm:Main thm}. The main tool here will be provided by the Hardy-Littlewood circle method. Given any $m\in\ZZ$, let
$$
\int_0^1 e(mz)dz =
\begin{cases}
1,  & \mbox{if $m = 0$}, \\
0,  & \mbox{otherwise},
\end{cases}
$$
denote the delta function detecting when an integer $m=0$. Here, $e(\alpha)=\exp(2\pi i \alpha)$, a standard notation. Using this, we start by rewriting $\Sigma(P)$ as
\begin{equation}\label{eq:sigdef1}
\Sigma(P)=\int_{0}^{1}S(\alpha)d\alpha,
\end{equation}
where, $S(\alpha) $ is as defined in \eqref{eq:Szdef} 
is an exponential sum. Typically, one needs to estimate $S(\alpha)$ at $\alpha=a/q+z$, where $|z|<q^{-2}$. One of our key ingredients in removing the dependence on the spectral gap is provided by the uniform bounds for twisted averages appearing in \cite{FFT16} and \cite{Tanis_Vishe17}. When $q$ is large or when $z$ is very small ($|z|\leq q^{-2}P^{-2+o(1)}$), we use van der Corput differencing to lower the degree of $F$ along with the bounds in \cite{FFT16}, which would hand us Lemma \ref{lem:van der}. This bound itself is unfortunately not enough to remove the dependence on the spectral gap when $q,|z|$ are {\em mid-range}. Here, we use a novel degree lowering technique. Namely, we split the sum over $\x$ in \eqref{eq:Szdef} as $\x=\x_1+N\x_2$ where $N$ is approximately of size $ |z|^{-1/2} $. This choice means that the term $zF(\x_1)$ is bounded. For a fixed value of $\x_2$, we then consider the sum over $\x_1$. This trick allows us to lower the degree of $F$ in the exponential integral which typically arises after applying Poisson summation. This is the essence of Lemma 
\ref{lem:splitbound}.

Let us briefly compare our work with that of Ubis \cite{Ubis17}. The key bound in \cite{Ubis17} uses van der Corput differencing to bound the exponential integrals, which is analogous to the bound \eqref{eq:boundmajor} of Lemma \ref{lem:van der} here. Here, we must point out that the hypersurface $F(\x)=0$ is of co-dimension one, and therefore, the $n-1$ dimensional volume of the set $\{F(\x)=0:|\x|<P\}\sim P^{n-1}$. In this paper however, we are averaging over a sparser subset in this manifold as demonstrated by the counting estimate \eqref{eq:NFPdef}. This is one philosophical reason behind why we need to establish the bound in Lemma \ref{lem:splitbound} and why this problem is significantly harder to tackle.
\subsection{Acknowledgements} The author had the inspiration for this project while discussing the aforementioned work of Ubis \cite{Ubis17} with Kevin Hughes. We are thankful to him for introducing us to this work. It has also been very helpful to get the input of Kevin Hughes, Asaf Katz and Adri\'an Ubis on an earlier version of this paper. Their contribution is greatly acknowledged. The simpler argument in the diagonal case was pointed out to us by Ubis, and we thank him for this as well.

\section{Auxiliary results}
In this section, we will gather together various auxiliary lemmas necessary for us. 

\subsection{Bounds for smooth twisted horocyclic averages on $M_0=\SLR/\Gamma_0$} Here, for the sake of avoiding the complication of introducing a separate notation, throughout this section, given $\omega\in\C^\infty_c(\RR)$, and $f\in\C^\infty(M_0)$, we will use the same notations $S_{p,k}(\omega)$ and $\|f\|_{\L^p_k}$ to denote the corresponding Sobolev norms. These can be seen to be equal to those in \eqref{eq:normdef} and \eqref{eq:normdef1} in the special case when $n=1$.

The first result to be obtained below is a smooth variant of a twisted averages result \cite[Theorem 1.1]{FFT16}. 
\begin{lemma}\label{lem:FFT}
Let $\omega\in \C^\infty_c(a,b)$ be a smooth, compactly supported function on $\RR$, let $f\in \C^\infty(M_0)$ be a function of zero average and let $x_0$ be any point in $M_0$. Then there exists $0<\gamma<1$ depending only on the spectral gap of $M_0$ such that given any $P>1$, any $\ve>0$, and any $c\in\RR$, we have
\begin{equation}
\label{eq:aux1}
\begin{split}
&P^{-1}\left|\int \omega(t/P)f(u_0(t)x_0)e(ct)dt\right|\\&\ll_{N,\ve} (1+|b-a|)\log^{1/2}(P)S_{1,1}(\omega)\min\{\|f\|_{\L^2_{7+\ve}}|P|^{-1/6}(1+|c|^{-1/6}),\|f\|_{\L^2_{3+\ve}}|P|^{-\gamma}\}.
\end{split}
\end{equation}
\end{lemma}
\begin{proof}
We begin by applying integration by parts to obtain
\begin{align*}
\int \omega(t/P)f(u_0(t)x_0)e(ct)dt=P^{-1}\int_{aP}^{bP} \omega'(t/P)\int_{aP}^t f(u_0(z)x_0)e(cz)dzdt.
\end{align*}
When $|cP|>e$, an application of \cite[Theorem 1.1, Equation (5)]{FFT16} to the inner integral on the right hand side of the above equation implies that this term is
\begin{equation}\label{eq:casse1}
\begin{split}
&\ll P^{-1}\log(P)^{1/2}\int_{aP}^{bP} |\omega'(t/P)|(1+|c|^{-1/6})|t-aP|^{5/6}dt\\
&\ll \|f\|_{\L^2_{7+\ve}} P^{5/6}\log(P)^{1/2}\int_{a}^{b} |\omega'(t)|(1+|c|^{-1/6})|t-a|^{5/6}dt\\ &\ll |b-a|^{5/6} S_{1,1}(\omega)\log(P)^{1/2}P^{5/6}\|f\|_{\L^2_{7+\ve}}(1+|c|^{-1/6}).
\end{split}
\end{equation}
On the other hand, when $|cP|<e$, an application of a weaker bound obtained at the bottom of \cite[Page 1361]{FFT16} hands us a constant $\gamma$, depending on the spectral gap of $M_0$, such that 
\begin{equation}
|\int_{aP}^t f_0(u_0(z)x_0)e(cz)dz|\ll \|f\|_{\L^2_{3+\ve}}|t-aP|^{1-\gamma}
\end{equation} giving the remaining bound in \eqref{eq:aux1}, after following the same steps as in the derivation of \eqref{eq:casse1} and further noting that $|b-a|^{1-\gamma}+|b-a|^{5/6}\ll 1+|b-a| $.
\end{proof}
Note that the explicit dependence on $|b-a|$ in Lemma \ref{lem:FFT} is not necessary for our applications. While applying, our function $\omega$ will be assumed to be supported in an interval of size $\ll 1$. We now focus our attention to estimating averages of smooth twisted averages:
\begin{lemma}
\label{lem:fsum}Given any $f\in\C^\infty(M_0)$, any $x_0\in M_0$, any $1\leq q\leq P$, any function $\omega\in \C^\infty_c(a,b)$, any $c\in\RR$, and any $\ve>0 $ we have
\begin{equation}
\label{eq:fsum1}
\begin{split}
&\sum_{v\in\ZZ}|\int\omega(y/P)f(u_0(y)x_0)e((c-v/q)y)dy|\\&\ll (1+|b-a|)\log^{1/2}(P) \|f\|_{\L^2_{9+\ve}}S_{1,3}(\omega)P((1+\|qc\|P/q)^{-1/6}+qP^{-1/6}),
\end{split}
\end{equation}
where $\|x\|$ denotes the distance of a real number $x$ to the nearest integer.

Moreover, if $\int_{M_0} f(x)d\mu_{G_0}(x)=0$, then we may further have
\begin{equation}
\begin{split}
\label{eq:fsum2}
&\sum_{v\in\ZZ}|\int\omega(y/P)f(u_0(y)x_0)e((c-v/q)y)dy|\\&\ll (1+|b-a|)\log^{1/2}(P) \|f\|_{\L^2_{9+\ve}}S_{1,3}(\omega)P(P^{-\gamma}+qP^{-1/6}),\end{split}
\end{equation}
where $\gamma$ is the constant appearing in the statement of Lemma \ref{lem:FFT}.
\end{lemma}
\begin{proof}
 Let $S$ denote the sum under investigation, that is, let
\begin{equation}
\label{eq:Sdef}
S:=\sum_{v\in\ZZ}\left|\int\omega(y/P)f(u_0(y)x_0)e((c-v/q)y)dy\right|.
\end{equation}
We first begin by considering the special case when $\int_{M_0} f(x)d\mu_{G_0}(x)=0$, i.e., when $f$ is a zero average function.
When $q\leq |qc-v|$, we will apply integration by parts twice, followed by Lemma \ref{lem:FFT}, while in the range $1/2\leq |qc-v|<q$, Lemma \ref{lem:FFT} will be directly applied. To this end, given any non-negative integer $k$ and any $c_1\in\RR$, integration by parts $k$ times leads us to
\begin{equation}
\label{eq:intpart}
\begin{split}
&\left|\int\omega(y/P)f(u_0(y)x_0)e(c_1y)dy\right|\\&\ll_k |c_1|^{-k}\sum_{j=0}^k P^{-j}\left|\int\omega^{(j)}(y/P)(X^{k-j}f)(u_0(y)x_0)e(c_1y)dy\right|.
\end{split}
\end{equation}
Here, $X$ is as in \eqref{eq:lieder}, acts on $f$ via the explicit action $Xf(x):=\frac{\partial}{\partial_t}|_{t=0}f(u_0(t)x) $. Lemma \ref{lem:FFT} can now be employed to estimate the inner integrals on the right hand of the above expression to obtain
\begin{equation}\label{eq:c-vlarge}\begin{split}
&\left|\int\omega(y/P)f(u_0(y)x_0)e(c_1y)dy\right|\\&\ll_{N,\ve,k} (1+|b-a|) P\log^{1/2}(P)|c_1|^{-k} S_{1,k+1}(\omega)\min\{\|f\|_{\L^2_{7+k+\ve}}P^{-1/6}(1+|c_1|^{-1/6}),\|f\|_{\L^2_{3+k+\ve}}|P|^{-\gamma}\}.\end{split}
\end{equation}
When $|qc-v|\geq q$, we apply \eqref{eq:c-vlarge} with $k=2$ and $c_1=c-v/q$ and when $1/2\leq |qc-v|<q$, we again apply \eqref{eq:c-vlarge} with $k=0$ and $c_1=c-v/q$ to obtain
\begin{align}
\notag&((1+|b-a|)P\log^{1/2}(P))^{-1}\sum_{\substack{ v\in\ZZ\\|qc-v|\geq 1/2}}\left|\int\omega(y/P)f(u_0(y)x_0)e((c-v/q)y)dy\right|\\&\ll \|f\|_{\L^2_{9+\ve}}S_{1, 3}(\omega)P^{-1/6}\left(\sum_{\substack{ v\in\ZZ\\ 1/2\leq |qc-v|<q}}q^{1/6}|qc-v|^{-1/6}+\sum_{\substack{ v\in\ZZ\\q\leq|qc-v|}}q^2|qc-v|^{-2})\right)\notag
\\
&\ll \|f\|_{\L^2_{9+\ve}}S_{1,3}(\omega)P^{-1/6}q.\label{eq:c-vlarge2}
\end{align}
Similarly, when $|qc-v|=\|qc\|<1/2$, we will apply \eqref{eq:c-vlarge} with $k=0$ and $c_1=\|qc\|/q$ to obtain
\begin{equation}
\label{eq:vsmall}
\begin{split}
((1+|b-a|)P\log^{1/2}(P))^{-1}&\left|\int\omega(y/P)f(u_0(y)x_0)e(\|qc\|y/q)dy\right|
\\&\ll S_{1,1}(\omega)\left(\min\{\|f\|_{\L^2_{9+\ve}}|P/q|^{-1/6}\|qc\|^{-1/6},\|f\|_{\L^2_{3+\ve}}|P|^{-\gamma}\}\right).\end{split}
\end{equation}
Combing \eqref{eq:c-vlarge2} and \eqref{eq:vsmall} together, we establish the Lemma when $f$ is of zero average.

When $f$ is not of zero average, we start by writing $f=f_0+\int_{M_0} f(x)d\mu_{G_0}(x) $, where $f_0$ is now a function of zero average. Thus,
\begin{equation}
S\leq S_1+|\int f(x)d\mu_{G_0}(x)|S_2,
\end{equation}
where 
\begin{equation}
\begin{split}
S_1:=\sum_{v\in\ZZ}\left|\int\omega(y/P)f_0(u_0(y)x_0)e((c-v/q)y)dy\right| \textrm{ and }
S_2:=\sum_{v\in\ZZ}\left|\int\omega(y/P)e((c-v/q)y)dy\right|.
\end{split}
\end{equation}
$S_1$ can be bound by our analysis above. Note that $f=f_0+\int_{M_0} f(x)d\mu_{G_0}(x)$ is an orthogonal decomposition of $f$ with respect to the $\L^2$ norm, and therefore, for every $k\geq 0$, we must have $\|f_0\|_{\L^2_k}\ll \|f\|_{\L^2_k}$. As a result, $S_1$ can be bound by
\begin{equation}\label{eq:12}
S_1\ll (1+|b-a|)P\log^{1/2}(P)S_{1,3}(\omega)\|f\|_{\L^2_{9+\ve}}((1+\|qc\|P/q)^{-1/6}+qP^{-1/6}).
\end{equation}
On the other hand, the sum $S_2$ is simpler and can be bound via direct integration by parts using
\begin{align*}
&S_2=P\sum_{v\in\ZZ}\left|\int\omega(y)e(P(qc-v)y/q)dy\right|\\&\ll PS_{1,2}(\omega)((1+P\|qc\|/q)^{-1}+(P/q)^{-2}\sum_{|qc-v|\geq 1/2}|qc-v|^{-2})\ll PS_{1,2}(\omega)(1+P\|qc\|/q)^{-1}.
\end{align*}
Combining this bound with the one in \eqref{eq:12}, and further noticing that $|\int f(x) d\mu_{G_0}(x)|\leq \|f\|_{\L^2_{9+\ve}}$, we establish \eqref{eq:fsum1}.
\end{proof}

It should be noted that since $M_0$ is assumed to be compact, the bounds here are independent of the choice of $x_0$.

Let $f\in \C^\infty(M_0)$ be a smooth function. Given any $t\in\RR$, we will also need bounds for the Sobolev norms of the function $u_0(t)\cdot f(x):=f(u_0(t)x_0)$. In particular, we would like to make the dependence on $t$ more explicit. We recall the explicit action of the basis \eqref{eq:lieder} of the Lie algebra in \cite[eq (3.1)]{Tanis_Vishe17}, 
\begin{equation}\begin{split}X(u_0(t)\cdot f)&=u_0(t)\cdot(Xf)\\Y(u_0(t)\cdot f)&=u_0(t)\cdot((Y+tX)f)\\Z(u_0(t)\cdot f) &=u_0(t)\cdot((Z-2tY-t^2X)f).
\end{split}
\end{equation}
Using this explicit action, followed by induction, we are able to prove that for any monomial $X^{i_1}Y^{i_2}Z^{i_3}$, of order $k=i_1+i_2+i_3$, we must have
\begin{equation*}
X^{i_1}Y^{i_2}Z^{i_3}(u_0(t)\cdot f)=\sum_{D\in \scrO_k} p_D(t)D(u_0(t)\cdot f),
\end{equation*}
where $p_D$ are polynomials of degree at most $2k$, with integer coefficients only depending on $i_1,i_2$ and $i_3$. Summing over all such monomials, and using the fact that action of $u_0(t)$ preserves the $\L^2$ norms of functions, for any $s\in\ZZ_{\geq 0}$ we have
\begin{equation}
\label{eq:n(t)f Bound}
\|u_0(t)\cdot f\|_{\L^{2}_{s}}\ll |t|^{2s}\|f\|_{\L^2_s}.
\end{equation}
Upon interpolation, this bound can be extended to be true for all $s\in\RR_{\geq 0}$. 

\subsection{A lattice sum bound} In the proof of Lemma \ref{lem:splitbound}, we will need a bound for the following lattice sum, which we derive next:
\begin{lemma}\label{lem:fdisc}
Let $L$ be a fixed invertible $n\times n$ matrix with $\ZZ$ entries and  let  $1\leq P,H$ be real numbers satisfying $0\leq H\leq P$. Then, given any $0<|z|<1$, any $0<C$ and any $0<\delta<1$,
\begin{align*}
\sum_{0\leq y_i\leq P}\prod_{i=1}^{n}((1+H\|z(L\y)_i\|)^{-\delta}+C)\ll_{L} P^n\prod_{i=1}^{n}(1/P+|z|+H^{-\delta}+(H|z|P)^{-\delta}+C).
\end{align*}
\end{lemma}
\begin{proof}
The bound is obvious if $1\ll |z|\ll 1$. So it is enough to assume that $0< |z|< 1/2$, say. By changing the variables to $\z=L\y$, it is enough to bound
\begin{align*}
\prod_{i=1}^{n}\sum_{ |z_i|\leq |L|P}((1+H\|zz_i\|)^{-\delta}+C).
\end{align*}
To bound the above expression, without loss of generality, we may assume that $z$ is positive. Let $N$ denote the nearest integer to $1/z$, which means $|N-1/z|\leq 1/2$. Moreover, since $0<z<1/2$, $N\geq 2$ and therefore, $|z-1/N|\leq z/(2N)<1/N^2$. We now write $z=1/N+z'$, where $|z'|<1/N^2$. We begin by noting that for any real number $r$, and for all but at most one integer $x$ satisfying $|x|<N/2$, we must have
\begin{align}\label{eq:extra}
\|r+zx\|=\|r+x/N+xz'\|\gg \|r+x/N\|,
\end{align}
since $|xz'|<1/(2N)$. Since, $L$ is assumed to be fixed throughout, our constants are free to depend on it, and therefore it is enough to look at
\begin{align*}
\prod_{i=1}^{n}\sum_{ -P\leq z_i\leq P}((1+H\|zz_i\|)^{-\delta}+C).
\end{align*}
If $P\geq N/2$, we begin by writing $z_i=z_{i,1}+\lceil N/2\rceil z_{i,2}$, where $|z_{i,1}|<N/2$. In the light of our observation \eqref{eq:extra}, for a fixed $i$,  
\begin{align*}
&\sum_{|z_i|\leq P}((1+(H\|zz_i\|))^{-\delta}+C)\ll PC+\sum_{0\leq |z_{i,2}|\leq P/N}\sum_{0\leq |z_{i,1}|<N/2}(1+H\|zz_{i,1}+z\lceil N/2\rceil z_{i,2}\|)^{-\delta}\\&\ll PC+\sum_{0\leq |z_{i,2}|\leq P/N}(1+\sum_{0\leq |z_{i,1}|<N/2}(1+H\|z_{i,1}/N+z\lceil N/2\rceil  z_{i,2}\|)^{-\delta})\\&\ll PC+\sum_{0\leq |z_{i,2}|\leq P/N}(1+\sum_{0\leq |z_{i,1}|<N/2}(1+|Hz_{i,1}/N|)^{-\delta})\\
&\ll P/N(1+NH^{-\delta})+PC\ll P(1/N+H^{-\delta}+C).
\end{align*}
On the other hand if $P<N/2$, then
\begin{align*}
&\sum_{0\leq z_i\leq P}((1+(H\|zz_i\|))^{-\delta}+C)\ll \sum_{0\leq z_i\leq P}((1+(H|z_i/N|))^{-\delta}+C)\\&\ll PC+1+\sum_{0<|z_i|\leq P}(H|z_i/N|)^{-\delta}\ll 1+PC+H^{-\delta}N^{\delta}P^{1-\delta}\\
&\ll P(1/P+(HP/N)^{-\delta}+C).
\end{align*}
Therefore,
\begin{align*}
&\prod_{i=1}^{n}\left(\sum_{ |z_i|\leq P}((1+H\|zz_i\|)^{-\delta}+C)\right)\ll P^n\prod_{i=1}^{n}(1/P+|z|+(HP|z|)^{-\delta}+H^{-\delta}+C),
\end{align*}
which implies the lemma.
\end{proof}
\section{Exponential sum estimates}
In this section, we will assume that $f$ and $w$ satisfy \eqref{eq:fprod} and \eqref{eq:wdef} respectively. Throughout, let $x_0\in M$ be an arbitrary point and let
\begin{align*}
x_0=(x_{0,1},...,x_{0,n}), \,\,\,\textrm{ where }x_{0,i}\in M_0\textrm{ for }i=1,...,n.
\end{align*} Given any $P>1$ and and any $\alpha\in\RR$, our prime focus in this section will be to establish bounds for the exponential sum $S(\alpha)$ defined in \eqref{eq:Szdef}:
\begin{equation*}
S(\alpha):=\sum_{\substack{\x\in\ZZ^n}}w(\x/P)f(u(\x)x_0)e(\alpha F(\x)).
\end{equation*}

 We would need to estimate $S(\alpha)$ ``near'' a rational number $0\leq a/q<1$. Therefore throughout, let $\alpha=a/q+z$, where $|z|<q^{-2}$. We would need to bound $S(a/q+z)$ in two different ways, which will be our focus in this section. For our first bound, i.e. Lemma \ref{lem:splitbound}, we will begin by splitting the sum over $\x $ as $\x_1+N\x_2$, for a suitable choice of $N$, depending on $z$. For a fixed choice of $\x_2$, we will estimate the corresponding exponential sum separately, and gain from the fact that for most of the values of $\x_2$, we would be able to bound the exponential sum satisfactorily. The second bound (Lemma \ref{lem:van der}) will be provided by van der Corput differencing. The first bound will be useful to deal with mid-ranges of $z$ and the latter will be used to deal when $z$ is small or relatively large.
\begin{lemma}\label{lem:splitbound}
Let $ P\in \ZZ_{>0}$,  let $f\in\C^\infty(M)$ and $w\in\C^\infty_c((-1,1)^n)$ satisfying \eqref{eq:fprod} and \eqref{eq:wdef} respectively, and let $\alpha\in\RR$ satisfying $\alpha=a/q+z$, where $ |z|\leq q^{-1}Q^{-1}$ where $1\leq q\leq Q=P^\Delta$, say. Then, given any $0<\ve\ll_\Delta 1$ we have
\begin{equation}\label{eq:boundminor1}
|S(\alpha)|\ll_{\ve,\Delta} S_{\infty,3n}(w)\|f\|_{\L^\infty_{9n+\ve}} P^{n+\ve}(q^{n/2}(|z|+1/P)^{n/6}+q^{-n/2}(1+|Pz^{1/2}|)^{-n/6}).
\end{equation}
\end{lemma}
\begin{proof}
 Let $0<\ve< \Delta/2$ be a small, positive number. Let $N=\min\{\lfloor  P^{-\ve}|z|^{-1/2}\rfloor,P\}$. The condition $|z|\leq q^{-1}P^{-\Delta}$ and that $\ve<\Delta/2 $ implies that  $1\leq N$. We begin by splitting the sum over $\x$ in \eqref{eq:Szdef} into $O((P/N)^n)$ sums of length $N$ each. \cite[Lemma 2]{Heath-Brown96} hands us an elegant and smooth way of doing so. \cite[Lemma 2]{Heath-Brown96} gives us that for any $0<\delta\leq 1$, there is a smooth function $\omega_\delta$ satisfying
 \begin{equation}
 \label{eq:omegasplit}
 \omega(x)=\delta^{-1}\int \omega_\delta(\frac{x-y}{\delta},y)dy.
 \end{equation}
 The function $\omega_\delta$ further satisfies 
 \begin{equation}\label{eq:omegader}
| \partial_{x,y}^\beta \omega_\delta(x,y)|\ll_{\beta}S_{\infty,|\beta|}(\omega).
 \end{equation}
Moreover, for a fixed $y$, the $x$-support of $w_\delta(x,y)$, is contained in the set $\{|x|\leq 1\}$, and the support of $\omega_\delta(\frac{x-y}{\delta},y)$ is contained in the support of $\omega$ for every $y$. Since $\omega$ is supported in $(-1,1)$, this implies that $\omega_\delta$ is supported in the set $[-1,1]\times [-1-\delta,1+\delta]$.

Using our definition of the function $w$ in \eqref{eq:wdef}, we may then analogously obtain
 \begin{equation}
 \label{eq:wsplit}
 w(\x)=\delta^{-n}\int w_\delta(\frac{\x-\y}{\delta},\y)d\y,
 \end{equation}
where
\begin{equation}
w_\delta(\x,\y)=\prod_{i=1}^n\omega_\delta(x_i,y_i).
\end{equation}
Thus, for any $0<\delta\leq 1$, and any $\x\in\RR^n$, we have
\begin{align*}
 w(\x/P)&=\delta^{-n}\int w_\delta(\frac{\x}{P\delta}-\frac{\y}{\delta},\y)d\y=\int w_\delta(\frac{\x}{P\delta}-\y,\delta\y)d\y\\
 &=\sum_{\y_0\in\ZZ^n}\int_{|\y_1|<1/2} w_\delta(\frac{\x}{P\delta}-\y_0-\y_1,\delta(\y_0+\y_1))d\y_1\\
 &=\sum_{\y_0\in\ZZ^n}W_{\delta,\y_0}(\frac{\x-P\delta\y_0}{P\delta}),
\end{align*}
where 
\begin{equation}\label{eq:Wdef}
W_{\delta,\y}(\x)=\int_{|\y_1|<1/2} w_\delta(\x-\y_1,\delta(\y+\y_1))d\y_1.
\end{equation}
Since the support of $w_\delta$ is contained in the hypercube $[-1,1]^n\times [-1-\delta,1+\delta]^n$, the sum over $\y_0$ is contained in the set $|\y_0|\ll \delta^{-1}$ and for such $\y_0$'s the function $W_{\delta,\y_0}(\x) $ is supported in the set $\{|\x|< 3/2\}$.

We now choose $\delta=N/P$, where $N$ as chosen at the beginning of the proof. Using this choice of $\delta$, we thus arrive at
\begin{equation*}
S(\alpha)=\sum_{\y_0\in\ZZ^n}\sum_{\x\in\ZZ^n}W_{\delta,\y_0}(\frac{\x-N\y_0}{N})f(u(\x)x_0)e(\alpha F(\x)).
\end{equation*}

At this point we introduce $\z=\x-N\y_0$. The above expression can be rewritten as
\begin{equation}\label{eq:salphasplit}
S(\alpha)=\sum_{\y_0\in\ZZ^n}\sum_{\z\in\ZZ^n}W_{\delta,\y_0}(\frac{\z}{N})f(u(\z+N\y_0)x_0)e(\alpha F(\z+N\y_0)).
\end{equation}
Note that for a fixed value of $\y_0\in\ZZ^n$, the function $W_{\delta,\y_0}(\z)$ is a smooth function supported in the set $\{|\z|< 2\}$. Moreover, using the bounds on the derivatives of $\omega_\delta$ in \eqref{eq:omegader}, we further have
\begin{equation}\label{eq:Wderibound}
|\partial^\beta_\z W_{\delta,\y_0}(\z)|\ll_{\beta}S_{\infty,|\beta|}(w).
\end{equation}
 The sum over $\y_0$ is supported in the set $\{|\y_0|\ll P/N\}$.

Let $\alpha=a/q+z$, as given. We now make a further change of variables $\z=\z_0+q\z_1$ to write $S(\alpha)$ as:
\begin{align}\label{eq:Ssimpli}
S(\alpha)=\sum_{\y_0\in\ZZ^n}\sum_{0\leq \z_0<q}\sum_{\z_1\in\ZZ^n}W_{\delta,\y_0}(\frac{\z_0+q\z_1}{N})f(u(\z)x_1)e(\alpha F(\z_0+q\z_1+N\y_0))).
\end{align}
Here, the notation $0\leq \z_0<q$ mean that each co-ordinate of $\z_0$ is an integer between (and including) $0$ and $q-1$. Here, 
\begin{equation}
\label{eq:x1def}
x_1:=u(N\y_0)x_0.
\end{equation}
 We begin by noting that
\begin{align*}
&e((a/q+z)F(\z_0+q\z_1+N\y_0))\\&=e((a/q+z)(F(\z_0+q\z_1)+2N(L\y_0)\cdot (\z_0+q\z_1)+N^2F(\y_0))\\
&=e((a/q+z)N^2F(\y_0))e_q(a(F(\z_0)+2N(L\y_0)\cdot \z_0))e(z(F(\z_0+q\z_1)+2N(L\y_0\cdot(\z_0+q\z_1))),
\end{align*}
where $e_q(x):=\exp(2\pi i x/q)$ as is a standard notation. Recall here that $L$ is the $n\times n $ integer matrix defining $F$.
For now, we will treat $\y_0$ as fixed and concentrate on the exponential sum
\begin{equation*}
\begin{split}
S_1:=S_1(z,\y_0):=&\sum_{0\leq \z_0<q}e_q(a(F(\z_0)+2N(L\y_0)\cdot \z_0))\times\\&\sum_{\z_1\in\ZZ^n}W_{\delta,\y_0}(\frac{\z_0+q\z_1}{N})f(u(\z_0+q\z_1)x_1)e(z(F(\z_0+q\z_1)+2N(L\y_0\cdot(\z_0+q\z_1)).\end{split}
\end{equation*}

We may now apply Poisson summation formula to the sum over $\z_1$ to obtain
\begin{equation*}
S_1=q^{-n}\sum_{\v\in\ZZ^n}S_q(a,\v)I(z,-2NzL\y_0+\v/q),
\end{equation*}
where 
\begin{equation}\label{eq:sexpdef}
S_q(a,\v):=\sum_{\x\mod{q}}e_q(a(F(\x)+2NL\y_0\cdot \x)+\x\cdot\v),
\end{equation}
is a standard quadratic exponential sum and
\begin{equation}\label{eq:Izvdef}
I(z,\v):=\int W_{\delta,\y_0}(\z/N)f(u(\z)x_1)e(zF(\z)-\v\cdot\z)d\z,
\end{equation}
is the corresponding exponential integral. 

The exponential sum we encounter in \eqref{eq:sexpdef} is a standard quadratic exponential sum. A standard bound that leads to \cite[Lemma 25]{Heath-Brown96} hands us square root cancellations in the exponential sums for all $\v$'s.  This follows essentially from squaring and further changing the variable to $\x_3=\x_2-\x_1$:
\begin{equation*}\begin{split}
|S_q(a,\v)|^2&=|\sum_{\x\mod{q}}e_q(a(F(\x)+2NL\y_0\cdot \x)+\x\cdot\v)|^2\\
&=|\sum_{\x_1,\x_2\mod{q}}e_q(a(F(\x_2)-F(\x_1))+(2aNL\y_0+\v)\cdot (\x_2-\x_1))|\\
&\ll \sum_{\x_1\bmod{q}}|\sum_{\x_3\bmod{q}}e_q(a(F(\x_1+\x_3)-F(\x_1))+(2aNL\y_0+\v)\cdot \x_3) |\\
&\ll \sum_{\x_1\bmod{q}}|\sum_{\x_3\bmod{q}}e_q((aM\x_1+2aNL\y_0+\v)\cdot \x_3)|\\
&\ll q^n\{\x_1\bmod{q}:q\mid (2aM\x_1+2aNL\y_0+\v)\}\ll q^n \{\x\bmod{q}:q\mid 2M\x\}.
\end{split}
\end{equation*} 
Here, to get the last inequality, we have used that if $\y=\y_1$ and $\y_2$ are two solutions of $q\mid 2aM\y+2aNL\y_0+\v $, then their difference must satisfy $q\mid 2M(\y_1-\y_2)$. Using the Smith normal form for $M=SDT$, where $S,D,T $ are matrices with integer entries, where $S,T$ have determinant $\pm 1$ and $D$ is a diagonal matrix, we are able to obtain: $$\{\x_1\bmod{q}:q\mid 2M\x_1\}\ll_F 1.$$ To sum up, for any integer $q$, any $a$ satisfying $\gcd(a,q)=1$, and any $\v\in\ZZ^n$ we have
\begin{equation}
\label{eq:expsumbound}
|S_q(a,\v)|\ll_{F} q^{n/2},
\end{equation}
where the implied constant only depends on the discriminant of the form $F$. The reader may also refer to \cite[Lemma 2.5]{Vishe19} where \eqref{eq:expsumbound} is proved in the function field setting. A minor modification of this bound will work here.
We now turn to bounding the exponential integral. Note that the exponential integral we encounter here will turn out to be simpler than the typical quadratic exponential integral which shows up in the circle method considerations. This is due to the fact that we have truncated the the sum over $\x$ to ensure that the integral over $\z$ is over a box of smaller size. As a result, $|zF(\z)|\ll P^{-\ve}$, for all $|\z|\leq 2N $. We may now use a Taylor series expansion to write 
\begin{equation}\label{eq:taylor}
e(zF(\z))=e(F(z^{1/2}\z))=\sum_{|\beta|\leq k/\ve}c_{\beta} (z^{1/2}\z)^\beta+O_{k,\ve}(P^{-k}),
\end{equation}
where, the constants $c_\beta$ are absolutely bounded $$|c_{\beta}|\ll_{\beta} 1,$$ and given $\beta=(\beta_1,...,\beta_n)\in\NN^n$, and any vector $\z\in\RR^n$, $\z^\beta$ denote the monomial
\begin{align*}
\z^\beta:=\prod_{i=1}^n z_i^{\beta_i}.
\end{align*} 
In light of \eqref{eq:taylor}, assuming that $\log q\ll \log P$, we have
\begin{equation}\label{eq:s1sum}\begin{split}
|S_1|&=q^{-n}\sum_{|\beta|\leq k/\ve}|c_\beta|\sum_{\v\in\ZZ^n}S_q(a,\v)I_\beta(z,-2NzL\y_0+\v/q)|+O_{k,\ve}(S_{\infty,0}(w)\|f\|_{\L^\infty}N^nP^{-k})\\
&\ll_{k,\ve} q^{-n/2}\sum_{|\beta|\leq k/\ve}\sum_{\v\in\ZZ^n}|I_\beta(z,-2NzL\y_0+\v/q)|+O_{k,\ve}(S_{\infty,0}(w)\|f\|_{\L^\infty}N^nP^{-k}),
\end{split}
\end{equation}
where $I_\beta$ is the exponential integral:
\begin{equation}\label{eq:Ibetadef}
I_\beta(z,\v):=\int (z^{1/2}\z)^\beta W_{\delta,\y_0}(\z/N)f(u(\z)x_1)e(-\v\cdot\z)d\z.
\end{equation}
 Note here that $|z|^{1/2}<1/N$, and therefore $|z^{1/2}N|< 1$. We next write
\begin{equation}\label{eq:Ibetadef'}
\begin{split}
I_\beta(z,\v)&=(z^{1/2}N)^{|\beta|}\int (\z/N)^\beta W_{\delta,\y_0}(\z/N)f(u(\z)x_1)e(-\v\cdot\z)d\z\\
&=(z^{1/2}N)^{|\beta|}\int  W_{\delta,\y_0,\beta}(\z/N)f(u(\z)x_1)e(-\v\cdot\z)d\z,
\end{split}
\end{equation}
where $W_{\delta,\y_0,\beta}(\z):=\z^\beta W_{\delta,\y_0}(\z)$ is a smooth function whose derivatives are $\ll$ those of $W_{\delta,\y_0}$ and further applying \eqref{eq:Wderibound} we have
\begin{equation}
|\partial^{\beta'}_\z W_{\delta,\y_0,\beta}(\z)|\ll_{\beta',\beta} S_{\infty, |\beta'|}(w).
\end{equation}
The main advantage of the Taylor expansion in \eqref{eq:taylor} is that the integral in \eqref{eq:Ibetadef'} now splits as a product of $n$ separate one dimensional integrals. We may now invoke Lemma \ref{lem:fsum} to bound each of these one dimensional integrals. We thus end up with
\begin{equation}\label{eq:s1bound}\begin{split}
&\sum_{\v\in\ZZ^n}|I_\beta(z,-2NzL\y_0+\v/q)|\\&\ll_{\ve}N^{n+\ve}S_{\infty,3n}(w)\prod_{i=1}^n\|f_i\|_{\L^\infty_{9+\ve}}\left((1+\|-2qzN(L\y_0)_i\|N/q)^{-1/6}+qN^{-1/6}\right).
\end{split}
\end{equation}
Here, since $M_0$ is compact, we have used the $\L^\infty$ bound to replace the $\L^2$ norm, and similarly used \eqref{eq:Wderibound} to bound the norm of $W_{\delta,\y_0}$ appearing there. Moreover, using \eqref{eq:fprod}, we may replace  $\prod_{i=1}^n\|f_i\|_{\L^\infty_{9+\ve}} $ simply by $\|f\|_{\L^\infty_{(9+\ve)n}}$. Substituting the bound in \eqref{eq:s1bound} to \eqref{eq:s1sum} and further summing over $\y_0$ in \eqref{eq:Ssimpli}, we obtain
\begin{equation*}
\begin{split}
|S(\alpha)|&\ll_{k,\ve}S_{\infty,3n}(w)\|f\|_{\L^\infty_{(9+\ve)n}}\times\\&\left(\sum_{\substack{\y_0\in\ZZ^n\\ |\y_0|\ll P/N}}N^{n+\ve}q^{-n/2}\prod_{i=1}^n\left((1+\|-2qzN(L\y_0)_i\|N/q)^{-1/6}+qN^{-1/6}\right)+P^{n-k}\right).
\end{split}
\end{equation*}
Since we are free to choose $k$, we may henceforth choose $k=n$. Therefore, note that the first term in the above equation is always dominant in this case and hence the term $P^{n-n}=1$ can be disregarded. When $N=P$, we simplify the above bound to get
\begin{equation*}
\begin{split}
|S(\alpha)|&\ll_{\ve}S_{\infty,3n}(w)\|f\|_{\L^\infty_{(9+\ve)n}}\sum_{\substack{\y_0\in\ZZ^n\\ |\y_0|\ll P/N}}N^{n+\ve}q^{-n/2}\left(1+qN^{-1/6}\right)^n\\
&\ll_{\ve}S_{\infty,3n}(w)\|f\|_{\L^\infty_{(9+\ve)n}}P^{n+\ve}(q^{-n/2}+q^{n/2}P^{-n/6}).
\end{split}
\end{equation*}
On the other hand, when $N=O(|z|^{-1/2}P^{-\ve})<P$, we may employ Lemma \ref{lem:fdisc} to obtain
\begin{equation*}
\begin{split}
|S(\alpha)|&\ll_{\ve}S_{\infty,3n}(w)\|f\|_{\L^\infty_{(9+\ve)n}} P^{n+\ve}q^{-n/2}(N/P+|qNz|+|N/q|^{-1/6}+|zNP|^{-1/6}+qN^{-1/6})^n\\
&\ll_{\ve}  P^{n+(n+1)\ve}S_{\infty,3n}(w)\|f\|_{\L^\infty_{(9+\ve)n}}(q^{-n/2}|z^{1/2}P|^{-n/6}+q^{n/2}|z|^{n/6}).
\end{split}
\end{equation*}
Combining these two bounds, and choosing an $\ve\ll_{\Delta, n}1$, we get \eqref{eq:boundminor1}.
\end{proof}

The above bound would need to be supplemented by a standard van der Corput bound, which we will obtain in the following lemma:
\begin{lemma}\label{lem:van der}
Let $\alpha=a/q+z$, where $1\leq q\leq P$, $\gcd(a,q)=1$ and $|z|\leq 1/q^2$. Then for all $0<\ve\ll 1$, we have
\begin{align}\label{eq:boundminor2}
|S(\alpha)|
&\ll_\ve  S_{\infty,3n}(w) \|f\|_{\L^\infty_{9n+\ve}} P^{n+\ve}(q^{-1/2}+(P/q)^{-1/228})^{n}.
\end{align}
Moreover, there exists $\gamma_1:=\gamma_1(\Gamma_0)$ such that for any $\alpha$ as before, we have
\begin{align}\label{eq:boundmajor}
|S(\alpha)|\ll S_{\infty,3n}(w) \|f\|_{\L^\infty_{9n+1}}P^{n-\gamma_1} .
\end{align}
\end{lemma}
\begin{proof}

 We start by noticing that for any $H\in\ZZ_{>0}$,
\begin{align*}
H^{n}S(\alpha):=\sum_{\x}\sum_{\substack{0\leq \h< H}}G(\x+\h),
\end{align*}
say, where 
\begin{equation}\label{eq:Gdef}G(\x)=w(\x/P)f(u(\x)x_0)e(\alpha F(\x)).\end{equation} Here $0\leq \h<H$ is a shorthand notation to denote that $h_i\in\ZZ$ satisfying $0\leq h_i<H$ for all $1\leq i\leq n$. Recall that $w$ is assumed to be supported in $(-1,1)^n$. Throughout, we will assume that $H\leq P/2$. Thus, the sum over $\x$ is supported in the set $-P\ll \x\ll P$. We may now use this fact and use Cauchy-Schwartz inequality for the sum over $\x$ to get
\begin{align*}
H^{2n}|S(\alpha)|^2&\ll P^{n}\sum_{\h_1,\h_2}\sum_{\substack{\x\in\ZZ^n\\0\leq \x+\h_1,\x+\h_2<P }}G(\x+\h_1)\overline{G(\x+\h_2)}\\
&\ll P^n\sum_{|\h|<H}N(\h)\sum_{\x\in\ZZ^n }G(\x+\h)\overline{G(\x)},
\end{align*}
where
 $$N(\h):=\#\{0\leq \h_1,\h_2<H: \h=\h_1-\h_2\}\leq H^n.$$
Thus,
\begin{align}\notag
|S(\alpha)|^2
&\ll P^nH^{-2n}\sum_{|\h|<H}N(\h)\sum_{\x\in\ZZ^n}w_\h(\x/P)f_{\h}(u(\x)x_0)e(\alpha(F(\x+\h)-F(\x)))\\
&\ll  P^nH^{-n}\sum_{|\h|<H}\left|\sum_{\x\in\ZZ^n}w_\h(\x/P)f_{\h}(u(\x)x_0)e(2(\alpha L\h)\cdot\x)\right|,\notag\\
&\ll  P^nH^{-n}\sum_{|\h|<H}\prod_{i=1}^n\left|\sum_{x_i\in\ZZ^n}\omega_{h_i}(\x/P)f_{i,h_i}(u_0(x_i)x_{0,i})e((2\alpha L\h)_ix_i)\right|,
\label{eq:vander2}
\end{align}
where since both $f$ and $w$ are assumed to be factorisable (see \eqref{eq:fprod} and \eqref{eq:wdef}), for any $y\in M$,
\begin{equation}
\label{eq:fhdefn}
f_\h(y):=f(u(\h)y)\overline{f(y)}:=\prod_{i=1}^n f_{i,h_i}(y_i):=\prod_{i=1}^n f_{i}(u_0(h_i)y_{i})\overline{f_{i}(y_i)},
\end{equation}
and
$$w_\h(\x):=w(\x+\h/P)\overline{w(\x)}:=\prod_{i=1}^n\omega_{h_i}(x_i):=\prod_{i=1}^n \omega(x_i+h_i/P)\overline{\omega(x_i)}. $$
Our main bound here will come from applying Poisson summation to the inner sums in \eqref{eq:vander2}, i.e., we obtain:
\begin{align}\notag
&\sum_{x_i\in\ZZ^n}\omega_{h_i}(x_i/P)f_{i,h_i}(x_i)e((2\alpha L\h)_ix_i)\\
&\ll  \sum_{v_i\in\ZZ}\left|\int \omega_{h_i}(x_i/P)f_{i,h_i}(u_0(x_i)x_{0,i})e(((2\alpha L\h)_i-v_i))x_i)dx_i\right|.\label{eq:vander1}
\end{align}
We now estimate the sum on the right hand side of \eqref{eq:vander1} via Lemma \ref{lem:fsum}. Therefore, for any $1\leq i\leq n$, and any $\ve>0$, we have
\begin{equation}\label{eq:hlargeb}
\begin{split}
&\sum_{v_i\in\ZZ}\left|\int \omega_{h_i}(x_i/P)f_{i,h_i}(u_0(x_i)x_{0,i})e(((2\alpha L\h)_i-v_i)x_i)dx_i\right|\\&\ll S_{1,3}(\omega_{h_i})\|f_{i,h_i}\|_{\L^2_{9+\ve}}P\log^{1/2}(P)((1+\|(2\alpha L\h)_i\|P)^{-1/6}+P^{-1/6}).
\end{split}
\end{equation}
We begin by bounding the derivatives of $f_{i,\h}$. Using the relation \eqref{eq:n(t)f Bound}, for any $k\in\ZZ_{\geq 0}$, and an element in the Lie algebra $D$ of order $k$,
\begin{equation*}
\|Df_{i,h_i}(x)\|_{\L^2}\ll (1+|h_i|)^{2k}\sum_{D_1,D_2:\textrm{ord}(D_1)+\textrm{ord}(D_2)=k}|\|D_1(f_i)(u_0(h_i)x_i)D_2(f_{i})(x_i)\|_{\L^2}
\end{equation*}
 As a result, an application of Cauchy-Schwartz inequality further implies
\begin{equation}
\label{eq:fhderi}
\begin{split}
\|f_{i,h_i}\|_{\L^2_k}&\ll (1+|h_i|)^{2k}\|f_i\|_{\L^4_{k}}^2\ll (1+|h_i|)^{2k}\|f_i\|_{\L^\infty_{k}}^2.
\end{split}
\end{equation}
Upon interpolation, this bound can be assumed to be true for all $k\in\RR_{\geq 0}$. Similarly, 
\begin{equation}
S_{\infty,3}(\omega_{h_i})\ll S_{\infty,3}(\omega)^2.
\end{equation}
Substituting  \eqref{eq:fhderi} back in \eqref{eq:hlargeb}, we get
\begin{equation}\label{eq:hlargeb1}
\begin{split}
&\sum_{v_i\in\ZZ}\left|\int \omega_{h_i}(x_i/P)f_{i,h_i}(u_0(x_i)x_{0,i})e(((2\alpha L\h)_i-v_i)x_i)dx_i\right|\\&\ll H^{18}S_{\infty,3}(\omega)^2\|f_i\|_{\L^\infty_{9+\ve}}^2P^{1+\ve}((1+\|(2\alpha L\h)_i\|P)^{-1/6}+P^{-1/6}).
\end{split}
\end{equation}
The above expression holds for $\ve$ small enough. Note that since $H\ll P$, the extra powers of $H^\ve$ have been absorbed into the term $P^\ve$. When $(L\h)_i=0$, the above bound is rather wasteful. In this case, we bypass Poisson summation and directly use the following bound:
\begin{equation}
\label{eq:Mh=0}
|\sum_{x_i\in\ZZ}\omega_{h_i}(x_i/P)f_{i,h_i}(u_0(x_i)x_{0,i})|\ll \|f_{i,h_i}\|_{\L^\infty}\sum_{x_i\in\ZZ}|\omega_{h_i}(x_i/P)|\ll P\|f_{i}\|_{\L^\infty}^2 S_{\infty,0}(\omega)^2.
\end{equation}
Therefore, for $\ve>0$ small enough, we have
\begin{align*}
&\sum_{|\h|<H}\left|\sum_{\x\in\ZZ^n}w_\h(\x/P)f_{\h}(u(\x)x_0)e(2(\alpha L\h)\cdot\x)\right|\\
&\ll P^{n+\ve}\sum_{|\h|<H}\prod_{i=1}^nS_{\infty,3}(\omega)^2\|f_i\|_{\L^\infty_{9+\ve}}^2(\delta_{(L\h)_i\neq 0}H^{18}((1+\|(2\alpha L\h)_i\|P)^{-1/6}+P^{-1/6})+\delta_{(L\h)_i=0})\\
&\ll S_{\infty, 3n}(w)^2 \|f\|_{\L^\infty_{(9+\ve)n}}^2 P^{n+\ve}\sum_{|\h|<|L|H}\prod_{i=1}^n(\delta_{h_i\neq 0}H^{18}((1+\|2\alpha h_i\|P)^{-1/6}+P^{-1/6})+\delta_{h_i=0}).
\end{align*}
Here, to obtain the last equation, we have made a change of variable to replace $L\h$ by $\h$.
Eventually, we will choose $H\leq q/(4|L|)$, which means that since $|z|<q^{-2}$,
$$|2zh_i|<1/(2q), \forall |h_i|<|L|H.$$
Thus, if $q\nmid h_i$, then
\begin{equation}
\label{eq:simple}
\|2\alpha h_i\|\gg 1/q. \end{equation}
However, if $|h_i|\leq |L|H\leq q|L|/(4|L|)=q/2$, then $q\mid h_i$ if and only if $h_i=0$. Therefore, when $h_i\neq 0$, where $|h_i|\leq |L|H$, we may use \eqref{eq:simple}.
Therefore, 
\begin{equation}\label{eq:final1}
\begin{split}
&|S(\alpha)|^2\\&\ll P^{2n+\ve}H^{-n}S_{\infty, 3n}(w)^2 \|f\|_{\L^\infty_{(9+\ve)n}}^2\sum_{|\h|<|L|H}\prod_{i=1}^n(\delta_{h_i\neq 0}H^{18}((1+\|2\alpha h_i\|P)^{-1/6}+P^{-1/6})+\delta_{h_i=0})\\
&\ll P^{2n+\ve}S_{\infty, 3n}(w)^2 \|f\|_{\L^\infty_{(9+\ve)n}}^2(H^{18}(|P/q|^{-1/6})+H^{-1})^n\\
&\ll P^{2n+\ve}S_{\infty, 3n}(w)^2 \|f\|_{\L^\infty_{(9+\ve)n}}^2\prod_{i=1}^n(H^{18}q^{1/6}P^{-1/6}+H^{-1}).\end{split}
\end{equation}
We now choose $$H=\min\{q/(4|L|),(P/q)^{1/114})\}, $$ to get
\begin{align}\notag
|S(\alpha)|
&\ll P^{n+\ve}(q^{-1/2}+(P/q)^{-1/228})^{n}S_{\infty, 3n}(w) \|f\|_{\L^\infty_{9n+\ve}}.
\end{align}
When $|z|$ and $q$ are small, we only hope to exploit from the sum over $i=1$ and apply the second bound in \eqref{eq:fsum2}. More explicitly, we begin with the following variant of \eqref{eq:vander2}
\begin{equation}
\begin{split}
&|S(\alpha)|^2\\&\ll P^{2n-1}H^{-n}\left(\prod_{i=2}^nS_{\infty,0}(\omega)^2\|f_i\|_{\L^\infty_0}^2\right)\sum_{|\h|<H}\left|\sum_{\substack{x_1\in\ZZ}}\omega_{h_1}(x_1/P)f_{h_1,1}(u_0(x_1)x_{0,1})e(2(\alpha L\h)_1x_1)\right|\\
&\ll P^{2n-1}H^{-n}\left(\prod_{i=2}^nS_{\infty,0}(\omega)^2\|f_i\|_{\L^\infty_0}^2\right)\sum_{|\h|<H}\sum_{\substack{v_1\in\ZZ}}\left|\int\omega_{h_1}(x_1/P)f_{h_1,1}(u_0(x_1)x_{0,1})e(((2\alpha L\h)_1-v_1)x_1)dx_1\right|\\
&\ll P^{2n+\ve}S_{\infty,3n}(w)^2\|f\|_{\L^\infty_{9n+\ve}}^2 H^{18}(P^{-\gamma}+P^{-1/6}).
\end{split}
\end{equation}
 Here, we have applied \eqref{eq:fsum2} to bound the sum over $x_1$. Note that the worse Sobolev norms appearing here are only chosen to match with our bounds in \eqref{eq:final1}. The second part of the lemma now follows from choosing $H=P^{\min\{\gamma,1/6\}/36}$, setting $\gamma_1=\min\{\gamma,1/6\}/78$ and by choosing $\ve\ll_{\gamma,n}1$.
\end{proof}

\section{Proof of Theorem \ref{thm:Main thm}}
Recall that \eqref{eq:sigdef} writes $\Sigma(P)$ as
\begin{equation*}
\Sigma(P)=\int_{0}^{1}S(\alpha)d\alpha,
\end{equation*}
where $S(\alpha)$ as in \eqref{eq:Szdef}. Let $1<Q<P$ be a parameter to be chosen later in due course. An application of Dirichlet approximation hands us:
\begin{equation}\label{eq:Diri}
(0,1)\subseteq \bigcup_{q=1}^{Q} \bigcup_{\substack{0\leq a<q\\ \gcd(a,q)=1}}\{|a/q-z|<(qQ)^{-1}\}.
\end{equation}
We now split $(0,1)$ into two regions which typically correspond to the major and minor arc regimes in the circle method setting. Let $\ve_0$ be a small parameter to be chosen in due course. We define 
\begin{equation}\label{eq:mdef}
\begin{split}
\fm_1&:=\bigcup_{q=1}^Q \bigcup_{\substack{0\leq a<q\\ \gcd(a,q)=1}} \{|a/q-z|<q^{-2}P^{-2+\ve_0}\} \textrm{ and }\\
\fm_2&:=\bigcup_{q=1}^Q \bigcup_{\substack{0\leq a<q\\ \gcd(a,q)=1}} \{q^{-2}P^{-2+\ve_0}\leq |a/q-z|<(qQ)^{-1}\}.
\end{split}
\end{equation}
When $\alpha\in \fm_1$ the bound from \eqref{eq:boundmajor} will suffice. On the other hand, when $\alpha\in \fm_2 $, we will use a combination of the bounds in \eqref{eq:boundminor1} and \eqref{eq:boundminor2}.
\begin{lemma}
\label{lem:minorbound}
For any $n\geq 481$ and any $0<\ve_0\leq 1/240$,  we have
\begin{equation*}
\int_{\fm_2}|S(\alpha)|d\alpha\ll P^{n-2-\ve_0/4}S_{\infty,3n}(w)\|f\|_{\L^\infty_{9n+1}}.
\end{equation*}
\end{lemma}
\begin{proof}
Let $Q=P^\Delta$ and let $0<\ve\ll_\Delta 1$ be an arbitrarily small number to be chosen later. We begin by combining bounds in \eqref{eq:boundminor1} and \eqref{eq:boundminor2} for any $\alpha=a/q+z$, where $|z|<(qQ)^{-1}$:
\begin{equation}\label{eq:firstb}
\begin{split}
&|S(a/q+z)|\ll_\ve P^{n+\ve}S_{\infty,3n}(w)\|f\|_{\L^\infty_{9n+\ve}}\times\\&\left(\min\{q^{1/2}|z|^{1/6},q^{-1/2}\}+q^{1/2}P^{-1/6}+q^{-1/2}(1+|Pz^{1/2}|)^{-1/6}+(P/q)^{-1/228}\right)^n\\
&\ll_\ve P^{n+\ve}S_{\infty,3n}(w)\|f\|_{\L^\infty_{9n+\ve}}\left(|z|^{1/12}+q^{1/2}P^{-1/6}+q^{-1/2}(1+|Pz^{1/2}|)^{-1/6}+(P/q)^{-1/228}\right)^n,
\end{split}
\end{equation}
where we have used a geometric mean to bound the first term inside the brackets on the right side.

We start first by examining the second last term:
\begin{equation}\label{eq:B1}
\begin{split}
&\sum_{q=1}^Q \sum_{\substack{0\leq a<q\\ \gcd(a,q)=1}}\int_{q^{-2}P^{-2+\ve_0}\leq|z|<(qQ)^{-1}}q^{-n/2}(1+|Pz^{1/2}|)^{-n/6}dz\\&\ll \sum_{q=1}^Q \sum_{\substack{0\leq a<q\\ \gcd(a,q)=1}}q^{-n/3}\int_{q^{-2}P^{-2+\ve_0}\leq|z|<(qQ)^{-1}}|qPz^{1/2}|^{-n/6}dz\\
&\ll \sum_{q=1}^Q q^{1-n/3}q^{-2}P^{-2}\int_{P^{\ve_0}\leq|z|<\infty}|z|^{-n/12}dz\ll P^{-2-\ve_0},
\end{split}
\end{equation}
as long as $n\geq 24$.
On the other hand,
\begin{equation*}
\begin{split}
(P/q)^{-n/228} 
+q^{n/2}P^{-n/6}\ll (P^{-1/228}Q^{1/228})^n+ P^{-n/6}Q^{n/2}.
\end{split}
\end{equation*}

Since $|z|\leq (qQ)^{-1}$, the term $|z|^{1/12}$ term may simply be bound by
\begin{equation}\label{eq:B4}
|z|^{1/12}\ll Q^{-1/12}.
\end{equation}
At this point, we choose $Q$ such that $Q^{1/12}=P^{1/228}/Q^{1/228}$, i.e., when $Q=P^{1/20}$ which means $\Delta=1/20$. For this choice of $Q$, 
\begin{equation}\label{eq:B2}
\begin{split}
(P/q)^{-n/228} 
+q^{n/2}P^{-n/6}+|z|^{n/12}\ll P^{-n/240}.
\end{split}
\end{equation}

Thus, as long as $n=481\geq 2\times 240+1$,
\begin{equation*}
\begin{split}
&\left(\min\{q^{1/2}|z|^{1/6},q^{-1/2}\}+(P/q)^{-1/228}+P^{-1/6}q^{1/2}\right)^n\ll P^{-2-1/240}.
\end{split}
\end{equation*}
Since the measure of $\fm_2$ is at most $1$, this leads to
\begin{equation}\label{eq:B7}
\begin{split}
&\int_{\fm_2}\left(\min\{q^{1/2}|z|^{1/6},q^{-1/2}\}+(P/q)^{-1/228}+P^{-1/6}q^{1/2}\right)^n d\alpha\ll P^{-2-1/240},
\end{split}
\end{equation}
as long as $n\geq 481$. 
Lemma \ref{lem:minorbound} now follows from combining bounds in \eqref{eq:B1} and \eqref{eq:B7} and further suitably choosing $\ve\leq \ve_0/4$.
\end{proof}
\begin{proof} (Proof of Theorem \ref{thm:Main thm}) In order to prove Theorem \ref{thm:Main thm}, it is enough to bound the contribution from $\alpha\in\fm_1$. Thus, using \eqref{eq:boundmajor}, for any $\ve_0$ we have
\begin{equation}\label{eq:majorbound}
\begin{split}
\int_{\fm_0} |S(\alpha)|d\alpha\ll P^{n-\gamma_1}\textrm{meas}(\fm_1)S_{\infty,3n}(w)\|f\|_{\L^\infty_{9n+1}}\ll  P^{n-2-\gamma_1+\ve_0}S_{\infty,3n}(w)\|f\|_{\L^\infty_{9n+1}}.
\end{split}
\end{equation}
Combining the results in Lemma \ref{lem:minorbound}  and \eqref{eq:majorbound}, and choosing $\ve_0=\min\{1/240,\gamma_1/2\}$ and setting $\gamma_0=\ve_0/4$, we establish Theorem \ref{thm:Main thm}.
\end{proof}
\begin{remark}\label{rk:diagonal} As mentioned in the introduction, the situation of diagonal forms is significantly easier. We will give a quick sketch of this argument here. In fact, it would be enough to have $F(\x)=F_1(\x_1)+F_2(\x_2)+F_3(\x_3)$, where $\x=(\x_1,\x_2,\x_3)$ where $\x_1,\x_2$ are at least two dimensional, and $\x_3$ being at least one dimensional. In this case, the exponential sum  $S(\alpha)$ naturally splits as
$$S(\alpha)=S_1(\alpha)S_2(\alpha)S_3(\alpha),$$
where $S_i$'s denote the corresponding exponential sums for the forms $F_i$ for $i=1,2,3$. We now apply the H\"older's inequality:
\begin{align*}
|\Sigma(P)|\leq \int_0^1 |S_1(z)S_2(z)S_3(z)|dz\leq \|S_3\|_{\L^\infty}\|S_1\|_{\L^2}\|S_2\|_{\L^2}.
\end{align*}
Using \eqref{eq:boundmajor}, we have $ \|S_3\|_{\L^\infty}\ll_{f} P^{n_3-\gamma_1}$, for some $\gamma_1>0$. Furthermore, for $i=1,2$, given any $\ve>0$, one may easily obtain
\begin{align*}
\int_0^1 |S_i(z)|^2dz&=\int_0^1 \sum_{\x_i,\x_i'\in\ZZ^{n_i}}W_i(\x_i/P)\overline{W_i(\x_i'/P)}f(u(\x_i)x_0)\overline{f(u(\x_i')x_0)}e(z (F(\x_i)-F(\x_i')))dz\\
&=\sum_{\x_i,\x_i'\in\ZZ^{n_i}}W_i(\x_i/P)\overline{W_i(\x_i'/P)}f(u(\x_i)x_0)\overline{f(u(\x_i')x_0)}\int_0^1 e(z (F(\x_i)-F(\x_i')))dz\\
 &\ll_{f} \#\{(\x_i,\x_i')\in\ZZ^{n_i}\times\ZZ^{n_i}:F_i(\x_i)-F_i(\x_i')=0, |\x_i|,|\x_i'|\ll P\}\ll_{f,\ve} P^{2n_i-2+\ve},
\end{align*}
where in the final bound we have used \cite[Theorem 2]{Heath-Brown02}, which applies as long as $n_1,n_2\geq 2 $. Combining these bounds, we end up with
\begin{align*}
|\Sigma(P)|\ll_{\ve,f} P^{n-2-\gamma_1+\ve},
\end{align*}
as long as $n_1,n_2\geq 2$ and $n_3\geq 1$.
\end{remark}
\section{Proof of Theorem \ref{thm:main thm1}}
We are now set to prove Theorem \ref{thm:main thm1}, which  will follow from Theorem \ref{thm:Main thm}. Throughout, we will assume that $n\geq 481$. We start by writing 
\begin{equation}
\sum_{\substack{\x\in\ZZ^n, |\x|<P\\ F(\x)=0 }}f(u(\x)x_0)=\sum_{\substack{\x\in\ZZ^n\\ F(\x)=0 }}W(\x/P)f(u(\x)x_0),
\end{equation}
where $W$ denotes the characteristic function of the hypercube $(-1,1)^n$. Since $F$ is supposed to have no local obstructions, the asymptotic formula \eqref{eq:NFPdef} implies that Theorem \ref{thm:main thm1} is equivalent to proving that
\begin{equation}
\lim_{P\rightarrow\infty}\frac{1}{P^{n-2}}\sum_{\substack{\x\in\ZZ^n\\ F(\x)=0 }}W(\x/P)f(u(\x)x_0)=0,
\end{equation}
for any continuous function $f$ of zero average.

 In order to invoke Theorem \ref{thm:Main thm}, we will approximate $W$ by a smooth function, and further approximate $f$ by a sum of factorisable functions of zero average. We start with the latter. Since $f$ is continuous and $M$ compact, using the
Stone–Weierstrass theorem for compact manifolds, given any $\ve>0$, we may write
\begin{align*}
f(g)=\sum_{i=1}^m h_i(g)+O(\ve),
\end{align*}
where $m$ may depend on $\ve$, and each $h_i$ is a smooth, factorisable function, that is, it is of the form
\begin{equation}\label{eq:factorizable}
h_i(g_1,...,g_n)=h_{i,1}(g_1)...h_{i,n}(g_n).
\end{equation}
Since $f$ is of zero average and $M$ is compact, we must further have
\begin{align*}
|\sum_{i=1}^m \int_{M}h_i(g)d\mu_{G}(g)|\ll \ve.
\end{align*}
Using this, we further reach:
\begin{align*}
f(g)=\sum_{i=1}^m h_i'(g)+O(\ve),
\end{align*}
where
\begin{equation}
h_i'(g)=h_i(g)-\int h_i(x)d\mu_G(x),
\end{equation}
is a function of zero average. Note that since $h_i$ is factorisable, $\int_M h_i(x)d\mu_G(x)=\prod_{j=1}^n \int_{M_0} h_{i,j}(x_j)d\mu_{G_0}(x_j)$. Now, we may next write
\begin{equation}\label{eq:h'pro}
\begin{split}
h_i'(g)&=\prod_{j=1}^n h_{i,j}(g_j)-\prod_{j=1}^n \int h_{i,j}(x_j)d\mu_{G_0}(x_j)\\&=\prod_{j=1}^n \left(\left((h_{i,j}(g_j)-\int h_{i,j}(x_j)d\mu_{G_0}(x_j)\right)+\int h_{i,j}(x_j)d\mu_{G_0}(x_j)\right)-\prod_{j=1}^n \int h_{i,j}(x_j)d\mu_{G_0}(x_j).
\end{split}
\end{equation}
Note that each function $h_{i,j}(g_j)-\int h_{i,j}(x_j)d\mu_{G_0}(x_j)$ is smooth and of zero average. After expanding out the product over $j$ in \eqref{eq:h'pro} and noticing that the constant term $\prod_{j=1}^n \int h_{i,j}(x_j)dx_j$ cancels out, we then write $h_i'$ as a sum of factorizable functions of zero average. Therefore, we may now assume that
\begin{equation}\label{eq:f-final}
f(g)=\sum_{i=1}^{m_1}\phi_i(g)+O(\ve),
\end{equation}
where $\phi_i$'s are factorisable functions of zero average. Note that the derivatives of $\phi_i$ also satisfy
\begin{equation}\label{eq:phideri}
\|\phi_i\|_{\L^\infty_k}\ll_{\ve, k, f} 1.
\end{equation}
 Therefore, we end up with
\begin{equation}\label{eq:smoothing}
\begin{split}
\sum_{\substack{\x\in\ZZ^n\\ F(\x)=0 }}W(\x/P)f(u(\x)x_0)&=\sum_{i=1}^{m_1}  \sum_{\substack{\x\in\ZZ^n\\ F(\x)=0 }}W(\x/P)\phi_i(u(\x)x_0)+O_f(\ve N_F(P))
\\
&=\sum_{i=1}^{m_1}  \sum_{\substack{\x\in\ZZ^n\\ F(\x)=0 }}W(\x/P)\phi_i(u(\x)x_0)+O_f(\ve P^{n-2}),
\end{split}
\end{equation}
using the asymptotic formula \eqref{eq:NFPdef}. Now let us focus on the sums corresponding to each $\phi_i$. In order to invoke Theorem \ref{thm:Main thm}, $W$ needs to be approximated by a smooth function. In order to do so, let $0<\delta<1$ be a parameter to be chosen in due course. Let $w$ be a smooth factorisable function of the type \eqref{eq:wdef} supported in $(-1,1)^n$. We may further assume that $w$ is a non-negative function taking values in the closed interval $[0,1]$, it takes value $1$ on the hypercube $(-1+\delta,1-\delta)^n$, and that the derivatives of $w$ satisfy
\begin{equation}
S_{\infty,k}(w)\ll \delta^{-k}.
\end{equation}
The asymptotic formula \eqref{eq:NFPdef} holds for any $P$, and therefore it hands us a constant $\gamma'>0$ depending only on $n$ and $F$ such that
\begin{equation*}
\begin{split}
&\sum_{\substack{\x\in\ZZ^n\\ F(\x)=0 }}W(\x/P)\phi_i(u(\x)x_0)\\&=\sum_{\substack{\x\in\ZZ^n\\ F(\x)=0 }}w(\x/P)\phi_i(u(\x)x_0)+O(\#\{\x\in\ZZ^n: (1-\delta)P\leq |\x|<P, F(\x)=0\})
\\
&= \sum_{\substack{\x\in\ZZ^n\\ F(\x)=0 }}w(\x/P)\phi_i(u(\x)x_0)+O(\|f\|_{\L^\infty}\delta P^{n-2})+O(\|f\|_{\L^\infty} P^{n-2-\gamma'}).
\end{split}
\end{equation*}
Since $\phi_i$ is a factorisable function of zero average, without loss of generality we can assume that it is of type \eqref{eq:fprod}. We are now able to apply Theorem \ref{thm:Main thm} to obtain
\begin{equation*}
\begin{split}
&\left|\sum_{\substack{\x\in\ZZ^n\\ F(\x)=0 }}W(\x/P)\phi_i(u(\x)x_0)\right|\ll_{f,\ve} \delta^{-9n}P^{n-2-\gamma_0}+\delta P^{n-2}+P^{n-2-\gamma'}.
\end{split}
\end{equation*}
At this point, we choose $\delta=P^{-\gamma_2}$, where $\gamma_2=\min\{\gamma', \gamma_0/(9n+1)\} $, and combine this bound with that in \eqref{eq:smoothing} to obtain
\begin{equation*}
\begin{split}
\left|\sum_{\substack{\x\in\ZZ^n\\ F(\x)=0 }}W(\x/P)f(u(\x)x_0)\right|\ll_{f} \ve P^{n-2}+C_\ve P^{n-2-\gamma_2},
\end{split}
\end{equation*}
where $C_\ve$ denotes a constant which depends only on $\ve, F, n$ and $\Gamma$. Since $\gamma_2$ is independent of $\ve$, for large enough $P$, we must have
\begin{equation*}
\begin{split}
P^{-(n-2)}\left|\sum_{\substack{\x\in\ZZ^n\\ F(\x)=0 }}W(\x/P)f(u(\x)x_0)\right|\ll_{f} \ve.
\end{split}
\end{equation*}
Since $\ve$ was chosen to be arbitrary, this establishes Theorem \ref{thm:main thm1}.
\bibliographystyle{plain}

\end{document}